\documentclass[12pt]{amsart}
\usepackage[top=1in, bottom=1in, left=1.25in, right=1.25in]{geometry}

\usepackage[colorlinks=true,pagebackref,hyperindex,citecolor=NavyBlue,linkcolor=BrickRed]{hyperref}
\usepackage{amsmath}
\usepackage{amsfonts}
\usepackage{amssymb}
\usepackage{amsthm}
\usepackage{stmaryrd}
\usepackage{enumitem}
\usepackage{color}
\usepackage[dvipsnames]{xcolor}
\usepackage{setspace}

\theoremstyle{theorem}
\newtheorem{theorem}{Theorem}[section]

\newtheorem{corollary}[theorem]{Corollary}
\newtheorem{lemma}[theorem]{Lemma}
\newtheorem{proposition}[theorem]{Proposition}

\newtheorem{theoremx}{Theorem}

\theoremstyle{definition}
\newtheorem{definition}[theorem]{Definition}

\newtheorem*{notation*}{Notation}
\newtheorem{setting}[theorem]{Setting}

\newtheorem{example}[theorem]{Example}

\newtheorem{remark}[theorem]{Remark}
\numberwithin{equation}{subsection}

\newcommand{\Proj}{\operatorname{Proj}}
\newcommand{\Spec}{\operatorname{Spec}}
\newcommand{\Ext}{\operatorname{Ext}}
\newcommand{\InjDim}{\operatorname{inj.dim}}
\newcommand{\Supp}{\operatorname{Supp}}
\newcommand{\Ht}{\operatorname{height}}	
\newcommand{\depth}{\operatorname{depth}}	
\newcommand{\Hom}{\operatorname{Hom}}	
\newcommand{\cc}{\operatorname{c}}
\newcommand{\cd}{\operatorname{cd}}
\newcommand{\ara}{\operatorname{ara}}
\newcommand{\codim}{\operatorname{codim}}

\renewcommand{\a}{\mathfrak{a}}
\renewcommand{\b}{\mathfrak{b}}
\renewcommand{\c}{\mathfrak{c}}

\newcommand{\m}{\mathfrak{m}}
\newcommand{\n}{\mathfrak{n}}
\newcommand{\p}{\mathfrak{p}}
\newcommand{\q}{\mathfrak{q}}
\renewcommand{\r}{\mathfrak{r}}

\newcommand{\PP}{\mathbb{P}}

\renewcommand{\k}{\Bbbk}
\newcommand{\V}{\mathbb{V}}

\newcommand{\onto}{\twoheadrightarrow}

\renewcommand{\(}{\left(}
\renewcommand{\)}{\right)}

\newcommand{\graph}{\Gamma}
\newcommand{\conncomp}{s}
\newcommand{\counter}{m}

\newcommand{\eltL}{(}
\newcommand{\eltR}{)}

\newcommand{\transition}[1]{#1}

\begin{document}

\title[Connectedness and Lyubeznik numbers]{Connectedness and Lyubeznik numbers}
\author[L.\ N\'u\~nez-Betancourt]{Luis N\'u\~nez-Betancourt}
\author[S.\ Spiroff]{Sandra Spiroff}
\author[E.\ E.\ Witt]{Emily E.\ Witt}


\begin{abstract}
We investigate the relationship between connectedness properties of spectra and the Lyubeznik numbers, numerical invariants defined via local cohomology.  
We prove that for complete equidimensional local rings, the Lyubeznik numbers characterize when connectedness dimension equals one. 
More generally, these invariants determine a bound on connectedness dimension.  
Additionally, our methods imply that the Lyubeznik number $\lambda_{1,2}(A)$ of the local ring $A$ at the vertex of the affine cone over a projective variety 
is independent of the choice of its embedding into projective space. 
\end{abstract}

\maketitle

\begin{center}
\dedicatory{\emph{Dedicated to Professor Gennady Lyubeznik on the occasion of his sixtieth birthday}}
\end{center}



\section{Introduction}

\subsection*{Connectedness dimension and local cohomology}
Given a complete local ring $A$ containing a field, with a separably closed residue field, write $A \cong R/I$, where $I$ is an ideal of $(R, \m, \k)$, an $n$-dimensional complete regular local ring.
Suppose further that $\dim(A) \geq 2$.
The second vanishing theorem of local cohomology (SVT)  states that the local cohomology module $H^{n-1}_I(R)$ vanishes if and only if the punctured spectrum $\Spec^\circ(A)$ of $A$ is connected \cite{Ogus73,PeskineSzpiro73}.  

Herein, our overarching goal is to better understand the relationship between local cohomology and connectedness properties of spectra.  
The SVT says that $H^{n-1}_I(R)$ vanishes if and only if $\cc(A) = 0$, where $\cc(A)$ is the  \emph{connectedness dimension} of $A$, the minimal dimension of a closed subset $Z$ of the spectrum of $A$ for which $\Spec(A) \setminus Z$ is a disconnected topological space.    
Inspired by this, it is natural to ask whether the connectedness dimension can be characterized by the vanishing of local cohomology more generally. 

We find that studying the \emph{Lyubeznik numbers} is advantageous in this regard.
With the notation above, given nonnegative integers $i$ and $j$, the Lyubeznik number $\lambda_{i,j}(A)$ of $A$ is the $i^\text{th}$ Bass number of the local cohomology module $H^{n-j}_I(R)$ with respect to $\m$, $\dim_{\k}\Ext^{i}_R(\k, H^{n-j}_I(R))$, which is finite \cite{Huneke, Lyubeznik93}.
The SVT demonstrates that the vanishing of $H^{n-1}_I(R)$ is independent of the representation of $A$ as $R/I$; the Lyubeznik numbers also do not depend on the representation \cite{Lyubeznik93}. 

If $A$ is equidimensional, the connectedness dimension of  $A$ is at least $\dim(A) - 1$ precisely when the highest Lyubeznik number $\lambda_{d, d}(A)$, where $d= \dim(A)$, is one  \cite{HochsterHuneke94, Walther01, LyuInvariants, Zhang07}. 
Moreover, $H^{n-1}_I(R)$ vanishes if and only if $\lambda_{0,1}(A)$ does, so that $\cc(A) \geq 1$ if and only if $\lambda_{0,1}(A) = 0$. 
We prove that the condition that $\cc(A) = 1$ is characterized by the Lyubeznik numbers $\lambda_{0,1}(A)$ and $\lambda_{1,2}(A)$, a more subtle 
requirement than a bound on cohomological dimension.  

\begin{theoremx}[Theorem \ref{lambda12Final: T}] \label{thm12intro: T}
Let $A$ be an equidimensional complete local ring containing a field, of dimension at least three, with a separably closed residue field.  
Then
\[
\cc(A) \geq 2 \text{ if and only if } \lambda_{0,1}(A) = \lambda_{1,2}(A) = 0.
\]
In other words, write $A \cong R/I$, where $I$ is an ideal of an $n$-dimensional complete regular local ring $(R, \m)$.
Then $\cc(A) \geq 2$ if and only if $H^{n-1}_I(R) = H^1_\m(H^{n-2}_I(R)) = 0$.  
\end{theoremx}

More generally, the connectedness dimension is bounded in terms of the vanishing of Lyubeznik numbers on the superdiagonal of the Lyubeznik table $\left[ \lambda_{i,j}(A) \right]_{0 \leq i, j \leq \dim(A)}$.

\begin{theoremx}[Theorem \ref{vanishingLyuNumConnDim: T}] \label{generalIntro: T}
Let $A$ be an equidimensional complete local ring containing a field, of dimension at least two, with a separably closed residue field.  
If $\lambda_{0,1}(A) = \lambda_{1,2}(A) = \cdots = \lambda_{i-1,i}(A) = 0$ for some $i < \dim(A)$, then $\cc(A) \geq i$.  
\end{theoremx}

The entries of the superdiagonal of the Lyubeznik table also determine a bound on the maximal number of connected components of $\Spec(A) \setminus Z$ among all closed subsets $Z$ of $\Spec(A)$ with at most a specified dimension; see Proposition \ref{countConnComp: P}.
 
 The technical framework designed to prove Theorems \ref{thm12intro: T} and \ref{generalIntro: T} 
 takes advantage of a family of graphs $\graph_t(A)$ that includes two that have previously been used to study the relationship between local cohomology and connectedness properties of spectra \cite{HunekeLyubeznik90, HochsterHuneke94, LyuInvariants, Zhang07}.  
Our method relies on comparing the connectedness dimension, and the Lyubeznik numbers, of the ring to the respective numerical invariants
modulo certain elements.

In particular, toward Theorem \ref{thm12intro: T}, analogous to the way that the highest Lyubeznik number counts the number of connected components of the Hochster-Huneke graph  \cite{LyuInvariants, Zhang07}, we prove and take advantage of the fact that the Lyubeznik number $\lambda_{1,2}(A)$ is completely determined by the number of connected components of certain supergraphs of the Hochster-Huneke graph (see Theorem \ref{MainTheorem: T}).

As a consequence of Theorem \ref{generalIntro: T}, in  Corollaries \ref{CDdepth: C} and  \ref{CDCD: C} we recover bounds comparing the connectedness dimension with cohomological dimension, and with depth \cite{FalNagoya, FalPRIMS, HochsterHuneke94, Varbaro09}.

\subsection*{Lyubeznik numbers of projective varieties}
Given an equidimensional projective variety  $X$ over a field $\k$, 
by choosing an embedding $X \hookrightarrow \PP^n_\k$, one can write $X=\Proj(R/I)$, where $I$ is a homogeneous ideal of the polynomial ring 
 $R=\k[x_0,\ldots,x_n]$. If $\m = \eltL x_0, \ldots, x_n \eltR$ is the homogeneous maximal ideal of $R$ and  $A = ( R/I )_\m$ is the local ring at the vertex of the affine cone over $X$,
Lyubeznik asked whether the Lyubeznik numbers $\lambda_{i,j}(A)$ are independent of the choice of embedding \cite{LyuSurveyLC}. 

The answer has been proven to be affirmative if $\k$ has prime characteristic \cite{WZ-projective}, or if $X$ is smooth \cite{SwitalaNonsingular}.
Additionally, the highest Lyubeznik number is always independent of the choice of embedding \cite{Zhang07}, as is  $\lambda_{0,1}(A)$  (cf.\,\cite[Proposition 3.1]{Walther01}).
The techniques developed to prove Theorem \ref{MainTheorem: T} enable us to prove that the same is true for $\lambda_{1,2}(A)$, and to explicitly characterize its value.

\begin{theoremx}[Theorem \ref{lambda12Proj: T}]
With the notation above, the Lyubeznik number  $\lambda_{1,2}(A)$ is independent of $n$, and of the choice of embedding of $X  \hookrightarrow \PP^n_\k$. 
\end{theoremx}

 In Theorem \ref{LyuProjBound: T}, we also establish a lower bound for the Lyubeznik numbers $\lambda_{i, i+1}(A)$, $1 < i < \dim(X)$, 
that is independent of the choice of embedding.  


\section{Connectedness of spectra and the graphs $\graph_\bullet$} \label{spectraGraph: S}

We begin this section by recalling some facts about connectedness of spectra.  
Next, we define a family of graphs associated to a ring, and point out a relationship between these graphs and connectedness properties of the spectrum of the ring.

The \emph{connectedness dimension} of a Noetherian ring $A$, denoted $\cc(A)$, is 
the minimal value of $\dim(X)$ among all closed sets $X$ of $\Spec(A)$ for which  $\Spec(A) \setminus X$ is disconnected. 
Given an ideal $\a$ of a Noetherian ring $A$, the topological space $\Spec(A) \setminus \V(\a)$ is disconnected if and only if there exist ideals $\b$ and $\c$ of $A$ such that $\sqrt{\b}, \sqrt{\c} \subsetneq \sqrt{\a}$, 
$\sqrt{\a \cap \b \cap \c} = \sqrt{0}$, and $\sqrt{\b + \c} \supseteq \sqrt{\a}$.
By convention, we assume that $\Spec(A) \setminus \Spec(A) = \varnothing$ is disconnected, so that 
if $A$ is an equidimensional local ring, then $0 \leq \cc(A) \leq \dim(A)$, and $\cc(A) = \dim(A)$ if and only if $A$ has one minimal prime \cite[19.1.10 and 19.2.5]{BrodmannSharpEd2}.

Suppose that $\p_1, \ldots, \p_s$ are the minimal primes of $A$.  Then the connectedness dimension of $A$ is the minimum of $\dim\( A /  ( \cap_{i \in \mathcal{S}} \p_i +  \cap_{j \in \mathcal{T}} \p_j ) \)$ among all subsets  $\mathcal{S}$ and $\mathcal{T}$ of $[s] = \{ 1, 2, \ldots, s\}$ such that $ \mathcal{S} \cup  \mathcal{T}  = [s]$ \cite[19.2.5]{BrodmannSharpEd2}

Recall that the \emph{arithmetic rank} of an ideal $\a$ of a ring, denoted $\ara(\a)$, is the least number of generators of an ideal with the same radical as $\a$.  
The following bound relating arithmetic rank and connectedness dimension is useful for us in the case that the ideal $\a$ is principal.  

\begin{remark}[Grothendieck's connectedness theorem] \label{connDimBound: R}
If $(A,\n)$ is an equidimensional complete local ring and $\a$ is a proper ideal of $A$, then 
\[
\cc(A/\a) \geq \min\{ \cc(A) , \dim(A) - 1 \} - \ara(\a)
\]
 \cite[Expos\'e XIII, Th\'eor\`eme 2.1]{GrothendieckSGA2},  \cite[19.2.10]{BrodmannSharpEd2}. 
 If $A$ has more than one minimal prime, then $\cc(A) < \dim(A)$ \cite[19.2.2]{BrodmannSharpEd2}.  Therefore, in this case, if $x \in \n$, then 
$\cc(A/\eltL x \eltR) \geq \cc(A) - 1$. 
\end{remark}

\transition{
We define a family of graphs associated to a ring, whose connectivity properties are closely related to connectedness properties of the ring's spectrum. 
}

\begin{definition} \label{graph: D}
Given an equidimensional local ring $A$ of dimension $d \geq 2$, and an integer $1 \leq t \leq d-1$, we define a graph $\graph_t(A)$ the following way:
\begin{enumerate}
\item The vertices of $\graph_t(A)$ are indexed by the minimal primes of $A$, and 
\item There is an edge between distinct vertices $\p$ and $\q$ if and only if \[\Ht_A(\p + \q) \leq t.\]
\end{enumerate}
\end{definition}

\noindent For any $1 \leq t < s \leq d - 1$, $\graph_t(A)$ is a subgraph of $\graph_{s}(A)$ with the same vertices. 
Therefore, if $\graph_t(A)$ is connected, then so is $\graph_s(A)$. 

\begin{remark}
If $A$ is an equidimensional complete local ring of dimension $d$ (so that $A$ is catenary), then  for every ideal $I$ of $A$, 
$
\Ht_A(I)=d-\dim(A/I). 
$
Accordingly, for $1 \leq t \leq d - 1$, there exists an edge between vertices $\p \neq \q$ of $\graph_t(A)$ precisely if $\dim\left( A/(\p + \q) \right) \geq d-t$.
\end{remark}

The following example illustrates some features of our results and is referenced in later sections.

\begin{example}\label{graphintro: E}
Given a field $\k$, let $I$ be the ideal $\eltL x,y \eltR \cap \eltL z,w \eltR \cap \eltL u,v\eltR$ of the ring $R = \k\llbracket x, y, z, w, u, v \rrbracket$.
Let $A = R/I$, hence $\dim(A) = 4$.  The three vertices of the graphs $\graph_t(A)$ correspond to the minimal primes $\eltL x,y \eltR, \eltL z,w \eltR$, and $\eltL u,v \eltR$ of $A$. The graph $\graph_1(A)$ has no edges, while the graphs $\graph_{2}(A)$ and $\graph_{3}(A)$ are complete.
\end{example}

The two graphs $\graph_1(A)$ and $\graph_{d-1}(A)$ have played an important role in studying cohomological dimension and connectedness properties of spectra.  
Suppose that $A$ is equidimensional and has a separably closed residue field. 
The graph $\graph_1(A)$ is the \emph{Hochster-Huneke graph} (or the \emph{dual graph}) of $A$,  often denoted $\Gamma_A$ in the literature \cite[Definition 3.4]{HochsterHuneke94}, \cite{Hartshorne62}. 
This graph is connected if and only if $\cc(A) \geq d - 1$ \cite[Theorem 3.6]{HochsterHuneke94}.  
Moreover, the highest Lyubeznik number $\lambda_{d,d}(A)$ equals the number of connected components of $\graph_1(A)$  \cite[Theorem 1.3]{LyuInvariants}, \cite[Main Theorem]{Zhang07}.

The fact that when $\dim(A) \geq 2$, $\graph_{d-1}(A)$ is connected if and only if $\Spec^\circ(A)$ is connected, i.e., $\cc(A) \geq 1$, is used in Huneke and Lyubeznik's proof of the second vanishing theorem of local cohomology \cite[Proof of Theorem 2.9]{HunekeLyubeznik90}. 
In fact, the number of connected components of $\graph_{d-1}(A)$ 
and $\Spec^\circ(A)$ coincide, and equal  $\lambda_{0,1}(A) - 1$ when $\dim(A) \geq 2$ (see \cite[Proposition 3.1]{Walther01} for the dimension two case). 
The graph $\graph_{d-1}(A)$ is the $1$-skeleton of a simplicial complex whose homology has connections with cohomological dimension and depth \cite[Theorem 1.1]{LyuLC}, \cite[Theorem 1.2]{ExtHar}, \cite[Proposition 3.6]{DaoTakagi}.


The following proposition, which relates the graphs $\graph_\bullet(A)$ to the connectedness of subsets of the spectrum of $A$, follows from \cite[Proposition 1.1]{Hartshorne62}, and generalizes \cite[Theorem 3.6]{HochsterHuneke94}.
We crucially rely on the fact that these graphs determine the connectedness dimension of the ring.

\begin{proposition} \label{graphConn: P}
Let $A$ be an equidimensional complete local ring of dimension $d \geq 2$.
Given $1 \leq t \leq d - 1$, 
$\graph_t(A)$ is connected if and only if $\cc(A) \geq d - t$; i.e., 
\[ \cc(A)= \max\{ i \geq 1  \;|\; \graph_{d-i}(A) \textup{ is connected\,}\} .\]
\end{proposition}

\begin{proof}
Since $X = \Spec(A)$ is a Noetherian topological space, $X \setminus Y$ is connected for all closed sets $Y$ of $X$ for which $\dim(Y) \leq d-t - 1$ (i.e., $\codim( Y, X) > t$)
if and only if for all minimal prime ideals $\q$ and $\r$ of $A$, there exists a sequence of minimal primes 
$
\q = \p_1, \p_2, \ldots, \p_{s} = \r
$
such that for all $1 \leq j < s$, $\dim\V(\p_j) \cap \V(\p_{j+1}) \geq d - t $, i.e., $\Ht_A(\p_j +\p_{j+1}) \leq t$  \cite[Proposition 1.1]{Hartshorne62}. 
\end{proof}

We call upon the following observation frequently, beginning with the proof of the corollary below.

\begin{remark}
Given ideals $\a$, $\b$, and $\c$ of a ring, 
$\sqrt{\a + (\b \cap \c)} = \sqrt{(\a + \b) \cap (\a + \c)}$.  
Indeed, we have that 
\begin{equation*} \label{inclusionsRadical: e}
(\a + \b)(\a + \c) \subseteq  \a + (\b \cap \c) \subseteq (\a + \b) \cap (\a + \c),
\end{equation*}
and since $(\a + \b)(\a + \c)$ and $(\a + \b) \cap (\a + \c)$ have the same radical, we obtain our desired conclusion by 
taking radicals. 

Then by an inductive argument, 
given ideals $\a_i$, $1 \leq i \leq u$, and $\b_j$, $1 \leq j \leq v$, 
if $\a = \cap_{i=1}^u \a_i$ and $\b = \cap_{j=1}^v \b_j$, then 
$\sqrt{ \a + \b } = \cap_{i=1}^u \cap_{j=1}^v  \sqrt{  \a_i + \b_j}$, so that 
\[\Ht_A( \a + \b ) = \min\{ \Ht_A(\a_i + \b_j) \mid 1 \leq i \leq u, 1 \leq j \leq v \}.\]  
\end{remark}

%

\begin{notation*}
Let $\# G$ denote the number of connected components of a graph $G$, and $\#X$ the number of connected components of a topological space $X$.  
\end{notation*}

Not only can the graphs $\graph_\bullet(A)$ detect the connectedness dimension of $A$, but they also determine a more refined numerical invariant of $A$. 

%

\begin{corollary} \label{numCC: C}
Given an equidimensional complete local ring $A$ of dimension $d \geq 2$, and an integer $1 \leq t \leq d - 1$,
\[
\#\graph_t(A) = \max\{  \#\(\Spec(A) \setminus X \)  \mid X \subseteq \Spec(A) \ \textup{closed},\, \dim(X) \leq d-t - 1 \}.
\]
\end{corollary}

\begin{proof}
When $\graph_t(A)$ is connected, the statement holds by Proposition \ref{graphConn: P}.  
Therefore, we address the case that the graph is disconnected.  

Given a ring $A$ satisfying our hypotheses, 
let $\conncomp = \#\graph_t(A)$, and let $\Sigma_1, \ldots, \Sigma_{\conncomp}$ denote the connected components of $\graph_t(A)$.  
For $1 \leq i \leq \conncomp$, let $\b_i$ denote the ideal of $A$ that is the intersection of the vertices of $\Sigma_i$.  

Let $\a = \cap_{1 \leq i < j \leq \conncomp} (\b_i + \b_j)$.  By writing each $\b_i$ as the intersection of its minimal primes, we see that $\Ht_A(\a) \geq t + 1$, 
since for $i \neq j$,  there is no edge between any minimal prime of $\b_i$ and any minimal prime of $\b_j$ in $\graph_t(A)$.
Thus,   $\dim(\V(\a)) \leq d - t - 1$.

We claim that  $X = \Spec(A) \setminus \V(\a)$  has exactly $\conncomp$ connected components, given by the closed subsets $X_i = \V(\b_i) \cap X$ of $X$, for $1 \leq i \leq \conncomp$.
Since for each $1 \leq i \leq \conncomp$, $\Ht_A(\b_i) = 0 < 2 \leq t + 1 \leq \Ht_A(\a)$, $\sqrt{\b_i} \not\supseteq \sqrt{\a}$, and $X_i \neq \varnothing$.

It is clear by definition of the ideals $\b_i$ that $\sqrt{\cap_i \b_i} = \sqrt{0}$, so that
 $\cup_{i=1}^\conncomp \V(\b_i) = \Spec(A)$, and
 $\cup_{i=1}^\conncomp X_i = X$. 
Moreover, for each $1 \leq i \leq \conncomp$, $(\cup_{j \neq i} X_j)  \cap X_i = \varnothing$ since for all such $i$, 
\[
 \sqrt{ \( \cap_{j \neq i} \b_i \)+ \b_j } = \sqrt{\cap_{j \neq i} (\b_i + \b_j) } \supseteq \sqrt{\cap_{1 \leq u  < v \leq \conncomp} (\b_u + \b_v) } = \sqrt{\a}.
\]

Therefore, to show that $X$ has $s$ connected components, it suffices to show that each $X_i$ is connected in $X$.
Since $X_i = \V(\b_i) \setminus \V(\b_i) \cap \V(\a) =\V(\b_i) \setminus \V(\b_i + \a)$, it suffices to show that $\Spec\(A/\b_i\) \setminus \V(\b_i + \a)$ is connected.  
Note that $A/\b_i$ is an equidimensional ring of dimension $d$, and $\graph_t\( A/\b_i \)$ is connected since it is isomorphic to the connected component of $\graph_t(A)$ with vertices corresponding to the minimal primes of $\b_i$. 
Therefore, by Proposition \ref{graphConn: P}, $\cc\(A/\b_i\) \geq d - t$. 
Since $\Ht_A(\a) \geq t + 1$, $\dim(A/\a) \leq d - t - 1$, so that $\Spec\(A/\b_i\) \setminus \V(\b_i + \a)$ must indeed be connected.

Our final step is to prove that $\#\(\Spec(A) \setminus \V(\a) \) \leq \conncomp$ for every ideal $\a$ of $A$ for which $\dim(A/\a) \leq d-t - 1$.  
Toward this, fix an ideal $\a$ of height at least $t + 1$.  
If $\Spec(A) \setminus \V(\a)$ is connected, the statement clearly holds.
Otherwise, if it has $\counter > 1$ connected components, then there exist ideals $\b_1, \cdots, \b_\counter$ for which 
$\sqrt{\a} \subseteq \sqrt{ \( \cap_{i \neq j} \b_i \) + \b_j } =  \cap_{i \neq j} \sqrt{ \b_i + \b_j }$. 
In particular, $\Ht_A(\b_i + \b_j) \geq \Ht_A(\a) \geq t+1$ for all $1 \leq i < j \leq \counter$, so that the minimal primes of $\b_i$ and of $\b_j$ must lie in distinct connected components of 
$\graph_t(A)$.  This means that $\counter \leq \conncomp$.

\end{proof}

\section{The graphs $\graph_\bullet$ modulo $x$}  \label{graphNZD: S}

In this section, we compare the graphs $\graph_\bullet(A)$ to $\graph_\bullet(A/\eltL x \eltR)$, where $x$ is a certain element of $A$, with the goal of finding such an $x$ for which these graphs have the same number of connected components (see Theorem \ref{CCsNZD: T}).

\begin{remark}
Let $(A, \n)$ be an equidimensional complete local ring containing a field of dimension $d \geq 1$. 
Fix $x\in \n$ that is in no minimal prime of $A$.
Given any minimal prime $\p$ of $\eltL x \eltR$, $\Ht_A(\p) \leq 1$ by Krull's principal ideal theorem, and $\Ht_A(\p) \geq 1$ by our choice of $x$.  
This means that $\Ht_A(\p)=1$, so that $\dim(A/\p)=d-1$. 
Hence, $A/\eltL x \eltR$ is also equidimensional.
\end{remark}

Our first goal is to show that if the graph associated to a ring is connected, then so is the graph modulo certain elements.  
Zhang proved this for the  Hochster-Huneke graph \cite[Proposition 2.2]{Zhang07}.
Our method is different, using Grothendieck's connectedness theorem and the relationship between our graphs and connectedness dimension, 
but we take advantage of Zhang's result by
comparing our more general graphs to the Hochster-Huneke graph.

\begin{proposition} \label{graphModuloNZD: P}
Let $(A, \n)$ be an equidimensional complete local ring containing a field of dimension $d \geq 3$, with a separably closed residue field. 
Fix $x\in \n$ that is in no minimal prime of $A$.  
If $1 \leq t \leq d-2$ and $\graph_t(A)$ is connected, then $\graph_t(A/\eltL x \eltR)$ is also connected.  
\end{proposition}


\begin{proof}
In the case that $A$ has more than one minimal prime, the result follows by Proposition \ref{graphConn: P} since $\cc(A/\eltL x \eltR) \geq \cc(A) -1$ by Grothendieck's connectedness theorem (see Remark \ref{connDimBound: R}).
On the other hand, suppose that $A$ has one minimal prime, $\p$.
Then the graphs $\graph_t(A)$ and $\graph_t(A/\p)$ each consist of one vertex, and since $\sqrt{\eltL x \eltR}   =  \sqrt{ \p + \eltL x \eltR}$, 
$\graph_t(A/\eltL x \eltR)$ and $\graph_t(A/ (\p +  \eltL x \eltR) )$ are also isomorphic graphs.
Thus, we can replace $A$ by $A/\p$ and assume $A$ is a domain, hence reduced.  
In this case, since the Hochster-Huneke graph $\graph_1(A)$ is connected,  $\graph_1(A/\eltL x \eltR)$ is also connected 
 \cite[Proposition 2.2]{Zhang07} (see also \cite[Proposition 3.5]{HNBPWInjDim}). 
Thus, the supergraph $\graph_t (A/\eltL x \eltR)$ with the same vertices is also connected.
\end{proof}

The following observation is useful in further relating the graphs associated to a ring with those associated to a quotient of the ring. 

\begin{lemma} \label{minimalPrimes: L}
Let $(A, \n)$ be an  equidimensional complete local ring, and fix $x \in \n$ in no minimal prime of $A$.  
Then every minimal prime of $\eltL x \eltR$ contains a minimal prime of $A$, and every minimal prime of $A$ is contained in some minimal prime of $\eltL x \eltR$.
\end{lemma}

\begin{proof}
If $\p_1, \ldots, \p_s$ denote the minimal primes of $A$, 
the first statement holds because $\sqrt{\eltL x \eltR} = \sqrt{\eltL 0\eltR + \eltL x \eltR} = \sqrt{ (\cap_{i=1}^s \p_i) + \eltL x \eltR} = \cap_{i=1}^s \sqrt{\p_i + \eltL x \eltR}$.

On the other hand, since $x$ is in no minimal prime $\p$ of $A$, $\dim(A/(\p + \eltL x \eltR)) = \dim(A/\p) - 1 = \dim(A) -1$  by Krull's principal ideal theorem, and $\Ht_A(\p + \eltL x \eltR) = 1$. 
 Fix a minimal prime $P$ of $\p + \eltL x \eltR$ of height one. 
Then $\eltL x \eltR \subseteq \p + \eltL x \eltR \subseteq P$, and since $\Ht_A\eltL x \eltR = 1$, $P$ is also a minimal prime of $\eltL x \eltR$.
\end{proof}

To obtain a partial converse of Proposition \ref{graphModuloNZD: P}, in which one places a stronger restriction on the choice element $x$, we define the following. 

\begin{definition} \label{primeSet: D}
Given a local ring $(A, \n)$, define 
\begin{equation*} \label{avoidanceSet: e}
\Xi(A) = \{ \p + \q \mid \p, \q \text{ minimal primes of } A \text{ for which } \sqrt{\p + \q} \neq \n \}.
\end{equation*}
If $\dim(A) > 0$, then every minimal prime $\p$ of $A$ is in $\Xi(A)$, by taking $\q = \p$.
\end{definition}

\begin{proposition} \label{graphModuloNZD2: P}
Let $(A, \n)$ be an equidimensional complete local ring containing a field of dimension $d \geq 3$.  
Fix $x \in \n$ that is in no minimal prime of any ideal in $\Xi(A)$ (see Definition \ref{primeSet: D}).
Then for any $1 \leq t \leq d-2$, if $\graph_t(A/\eltL x \eltR)$ is connected, then $\graph_t(A)$ is also connected.  
\end{proposition}

\begin{proof}
By Proposition \ref{graphConn: P}, it is enough to show that if $\cc( A/\eltL x \eltR ) \geq d - t - 1$, then $\cc(A) \geq d - t$; i.e., that if $\cc( A/\eltL x \eltR ) \geq m$ for some $1 \leq m \leq d - 2$, then $\cc( A/\eltL x \eltR ) \leq \cc(A) - 1$. 
Toward this, suppose that $\p_1, \ldots, \p_s$ are the minimal primes of $A$, and fix subsets $\mathcal{S}$ and $\mathcal{T}$ of $[s]$ for which  $\mathcal{S} \cup \mathcal{T} = [s]$ and
$\cc(A) = \dim\(A/ (\cap_{i \in \mathcal{S}} \p_i + \cap_{j \in \mathcal{T}} \p_j) \) \geq m$ for some $1 \leq m \leq d -2$.
Notice that this means $\sqrt{\p_i + \p_j} \neq \n$ for some $i \in \mathcal{S}$ and $j \in \mathcal{T}$, else $\cc(A) = 0$.  

By Lemma \ref{minimalPrimes: L}, for each $1 \leq i \leq s$, there is at least one minimal prime of $\eltL x \eltR$ that contains $\p_i$, and every minimal prime of $\eltL x \eltR$ contains some $\p_i$. 
Let $Q_i$ denote the intersection of all minimal primes of $\eltL x \eltR$ that contain $\p_i$.
Since $x$ is in no minimal prime of any ideal in $\Xi(A)$, 
$\Ht_A(\p_i + \p_j + \eltL x \eltR) = \Ht_A(\p_i + \p_j) + 1$ by Krull's principal ideal theorem as long as $\p_i + \p_j$ is not $\n$-primary, and
\begin{align*}
\dim\(A/ (\cap_{i \in \mathcal{S}} \p_i + \cap_{j \in \mathcal{T}} \p_j) \) 
&= \dim\(A/ (\cap_{i \in \mathcal{S}} \cap_{j \in \mathcal{T}} (\p_i +  \p_j ) ) \)  \\
&= \dim\(A/ (\cap_{i \in \mathcal{S}} \cap_{j \in \mathcal{T}} (\p_i +  \p_j + \eltL x \eltR) ) \)  + 1 \\
&=  \dim\(A/ (\cap_{i \in \mathcal{S}} (\p_i + \eltL x \eltR ) + \cap_{j \in \mathcal{T}} ( \p_j + \eltL x \eltR ) ) \)  + 1 \\
&\geq  \dim\(A/ (\cap_{i \in \mathcal{S}} Q_i + \cap_{j \in \mathcal{T}}  Q_j \)  + 1 \\
&\geq \cc(A/\eltL x \eltR ) + 1.
\end{align*}
We conclude that $\cc(A/\eltL x \eltR) \leq \cc(A) - 1$. 
\end{proof}

Propositions \ref{graphModuloNZD: P} and \ref{graphModuloNZD2: P} allow us to conclude that the graph associated to a ring is connected if and only those associated to the ring modulo certain elements are connected. To wit:

\begin{theorem} \label{connectedModNZD: T}
Let $(A, \n)$ be an equidimensional complete local ring containing a field of dimension $d \geq 3$, with a separably closed residue field.  
Fix $x \in \n$  that is in no minimal prime of any ideal in  $\Xi(A)$ (see Definition \ref{primeSet: D}).
Then for any $1 \leq t \leq d-2$,  $\graph_t(A)$ is connected if and only if $\graph_t(A/\eltL x \eltR)$ is connected.  
\end{theorem}

\begin{proof}
This is an immediate consequence of Propositions \ref{graphModuloNZD: P} and \ref{graphModuloNZD2: P}.  
\end{proof}

\transition{
Finally, we extend Theorem \ref{connectedModNZD: T} to address the setting in which the graphs need not be connected.  
}

\begin{theorem} \label{CCsNZD: T}
Let $(A, \n)$ be an equidimensional complete local ring containing a field of dimension $d \geq 3$, with a separably closed residue field.  
Fix $x \in \n$  that is in no minimal prime of any ideal in  $\Xi(A)$ (see Definition \ref{primeSet: D}).
Then for every $1 \leq t \leq d-2$, 
$
\#\graph_t(A) = \#\graph_t(A/\eltL x \eltR).
$
\end{theorem}

\begin{proof}
First notice that by Lemma \ref{minimalPrimes: L}, each minimal prime of $\eltL x \eltR$ is a minimal prime of $\p + \eltL x \eltR$, for some minimal prime $\p$ of $A$.
Fixing $1 \leq t \leq d-2$, this means that the vertices of $\graph_t(A/\eltL x \eltR)$ can be indexed by the ideals of the form $\q \, A/\eltL x \eltR$, where $\q$ is a minimal prime of $\p + \eltL x \eltR$, 
 for some minimal prime $\p$ of $A$.  

Suppose that $\# \graph_t(A) = \conncomp$, and let $\Sigma_1, \ldots, \Sigma_{\conncomp}$ denote the connected components of $\graph_t(A)$.
For $1 \leq i \leq \conncomp$, let $\a_i$ be the ideal of $A$ that is the intersection of the minimal primes of $A$ that correspond to the vertices of $\Sigma_i$.  
Then each $\Sigma_i$ is isomorphic to $\graph_t(A/\a_i)$, so by Proposition \ref{graphModuloNZD: P}, each $\graph_t(A/(\a_i + \eltL x \eltR))$ is also connected.  

By Lemma \ref{minimalPrimes: L},  $\graph_t( A/ \eltL x \eltR )$ is isomorphic to the graph $\graph_t\( A/ \cap_{i=1}^{\conncomp} (\a_i + \eltL x \eltR ) \)$, and if $i \neq j$, then no minimal prime of $\a_i + \eltL x \eltR$ is also a minimal prime of $\a_i + \eltL x \eltR$.  

Hence, to finish the proof, it suffices to show that 
there is no edge between vertices corresponding to any minimal prime of $\a_i + \eltL x \eltR$ and any minimal prime of $\a_j + \eltL x \eltR$ in this graph, for $1 \leq i < j \leq \conncomp$.  
However, if there were an edge, then $\graph_t\( A/ (\a_i \cap \a_j + \eltL x \eltR )\)$ would be connected, which would imply that 
$\graph_t\(A/ \a_i \cap \a_j \)$ is connected by Proposition \ref{graphModuloNZD2: P}, contradicting the fact that $\Sigma_i$ and $\Sigma_j$ are distinct connected components of $\graph_t(A)$.

We conclude that given $1 \leq i \leq \conncomp$, the induced subgraph of $\graph_t( A/ \eltL x \eltR )$ whose vertices 
correspond to minimal primes of $\p + \eltL x \eltR$ for $\p$ a fixed minimal prime of $A$ is a connected component of  $\graph_t( A/ \eltL x \eltR )$.
In particular, $\graph_t(A/\eltL x \eltR )$ has $\conncomp$ connected components.
\end{proof}


\section{Lyubeznik numbers modulo an element}  \label{LyuNZD: S}


Inspired by Zhang's proof that the highest Lyubeznik number counts the connected components of the Hochster-Huneke graph \cite{Zhang07}, we relate Lyubeznik numbers of a ring to those modulo a carefully chosen nonozerdivisor.  
Rather than comparing the highest Lyubeznik numbers, we compare those on the superdiagonal of the Lyubeznik table. 
The relationship can be more subtle for the Lyubeznik numbers we consider; i.e., we may not obtain an equality.   

Note the following observation on the support of local cohomology.  

\begin{lemma}[cf.\,{\cite[Lemma 2.4]{Zhang07}}]  
\label{nzdChoice: L}
Suppose that $(R, \m)$ is a regular local ring of dimension $n$, and let 
$I$ be an ideal of $R$ for which $R/I$ is equidimensional of dimension $d \geq 2$.
Fix an integer $1 \leq i \leq d- 1$,  and $r \in \m$ that satisfies the following property:  
\begin{quotation}
If $\Supp_R H^{n-i}_I(R) \neq \{ \m \}$, then 
$r$ is in no minimal prime of $\Supp_R H^{n-i}_I(R)$.  
\end{quotation}
\noindent Then $\dim \Supp_R H^0_{\eltL r \eltR} ( H^{n-i}_I(R) )\leq \max\{0,i-2\}$.
\end{lemma}

\begin{proof}
The result holds when $H^0_{\eltL r \eltR} ( H^{n-i}_I(R) ) = 0$.  Otherwise, we claim
\begin{equation}  \label{dimSuppHn2: e}
\dim \Supp_R H^{n-i}_I(R) \leq i-1.
\end{equation}
To see this, fix $\p \in  \Supp_R H^{n-i}_I(R)$, so that $(H^{n-i}_I(R))_{\p} = H^{n-i}_{I R_\p}(R_\p) \neq 0$.  
By way of contradiction, assume that $\Ht_R(\p) \leq n - i$.
If $\Ht_R(\p) < n - i$, then $H^{n-i}_{I R_\p}(R_\p)$ vanishes since $n-i > \dim (R_{\p})$.  
On the other hand, if $\Ht_R(\p) = n - i$, then since $\Ht_R(I) = n - d \leq n-i-1$ and  $R/I$ is equidimensional, 
$I R_\p$ cannot be $\p R_\p$-primary.
Again, we have that $H^{n-i}_{I R_\p}(R_\p) = 0$, now by the HLVT \cite[Theorem 3.1]{Hartshorne68}.  
In either case, we have arrived at a contradiction, and \eqref{dimSuppHn2: e} holds.

We proceed to show that $\dim \Supp_R H^0_{\eltL r \eltR} ( H^{n-i}_I(R) )\leq i-2$.  
This is clear if $\dim \Supp_R  H^{n-i}_I(R) \leq i-2$, since $H^0_{\eltL r \eltR} ( H^{n-i}_I(R) )$ is a submodule of $H^{n-i}_I(R)$.  
Otherwise, fix some $\p \in \Supp_R H^0_{\eltL r \eltR} ( H^{n-i}_I(R) )$.  
Then $\p \in \Supp_R H^{n-i}_I(R)$, so that $\p$ contains some minimal element $\q$ of $\Supp_R H^{n-i}_I(R)$.  
Moreover, $r \in \p$, but $r \notin \q$ by our choice of $r$. Then, by Krull's principal ideal theorem and \eqref{dimSuppHn2: e} we have that
\begin{equation}
\Ht_R(\p) \geq \Ht_R(\q + \eltL r \eltR) = \Ht_R(\q) + 1 \geq n-i+2;
\end{equation}
i.e., $\dim(R/\p) \leq i-2$.
\end{proof}

Our next goal is to relate the Lyubeznik numbers on the superdiagonal of the Lyubeznik table to those on the superdiagonal for a quotient modulo an element. 
We can choose an element $r \in \m$ as specified in Propositions \ref{lambda12: P} and \ref{lambdaBound: P} by prime avoidance, since each local cohomology module $H^{i}_I(R)$ has finitely many associated primes \cite[Corollary 3.6]{Lyubeznik93}.

\begin{proposition}\label{lambda12: P}
Suppose that $(R, \m)$ is a $n$-dimensional complete regular local ring containing a field.  
Let $I$ be an ideal of $R$ for which $R/I$ is equidimensional of dimension $d \geq 3$. 
Moreover, fix $r \in \m$ satisfying the following property: 
\begin{quotation}
 If $\Supp_R H^{n-2}_I(R) \neq \{ \m \}$, then 
$r$ is in no minimal prime of $\Supp_R H^{n-2}_I(R)$. 
\end{quotation} 
Then $\lambda_{0,1}(R/I) + \lambda_{1,2}(R/I) = \lambda_{0,1}(R/ ( I + \eltL r \eltR) )$.  
In particular, if the residue field of $R$ is separably closed and $\Spec^\circ(R/I)$ is connected, then 
$\lambda_{1,2}(R/I)  =  \lambda_{0,1}(R/ ( I + \eltL r \eltR) )$.
\end{proposition}

\begin{proof}
Let $J = I + \eltL r \eltR$ in $R$, and consider the long exact sequence \begin{equation} \label{LESsmallcase: e}
\cdots \to H^{n-2}_I(R) \to H^{n-2}_I(R)_r \to H^{n-1}_J(R) \to H^{n-1}_I(R) \to H^{n-1}_I(R)_r \to \cdots
\end{equation}
(see e.g., \cite[8.1.2]{BrodmannSharpEd2}).  
We obtain from \eqref{LESsmallcase: e} the following short exact sequences, for appropriate $R$-modules $K$, $M$, $N$, $P$, and $Q$:  
\begin{align} 
\label{SESeltOne: e} &0 \to K \to H^{n-2}_I(R) \to M \to 0 \\
\label{SESeltTwo: e} &0 \to M \to H^{n-2}_I(R)_r \to N \to 0  \\
\label{SESeltThree: e} &0 \to N \to H^{n-1}_J(R) \to P \to 0 \\ 
\label{SESeltFour: e} &0 \to P \to H^{n-1}_I(R) \to Q \to 0.
\end{align}

Notice that $K \cong H^0_{\eltL r \eltR}(H^{n-2}_I(R))$ and $P \cong  H^0_{\eltL r \eltR} (H^{n-1}_I(R))$. 
As the injective dimension of each of $K$ and $P$ is no more than the dimension of its support \cite[Theorem 3.4]{Lyubeznik93}, 
by Lemma \ref{nzdChoice: L}, each module is either zero or supported only at $\m$.
Hence, 
$H^j_\m( K ) = 0 \text{ for all } j  \geq 1$.  
Thus, the long exact sequence in local cohomology with support in $\m$ with respect to \eqref{SESeltOne: e} tells us that
$H^1_\m(M) \cong H^1_\m(H^{n-2}_I(R))$ and $H^2_\m(M) = H^2_\m(H^{n-2}_I(R))$, the latter of which vanishes since $R/I$ is equidimensional (see the proof of \cite[Property (4.4iii)]{Lyubeznik93}).

Moreover, as $r \in \m$, so that $H^j_\m( H^{n-2}_I(R)_r ) =0$ for all $j$, the long exact sequence with respect to \eqref{SESeltTwo: e} enables us to conclude that 
$H^1_\m(M)  \cong H^0_\m( N )$ and $H^2_\m(M)  \cong H^1_\m( N )$.
Consequently, 
\begin{equation} \label{conclusion1LES: e}
H^1_\m(H^{n-2}_I(R)) \cong H^0_\m( N )   
\end{equation}
and $H^1_\m(N) = 0$, the latter of which, combined with the long exact sequence with respect to \eqref{SESeltThree: e}, exhibits the following short exact sequence:
\begin{equation} \label{SES12: e}
0 \to H^0_\m(N) \to H^0_\m( H^{n-1}_J(R) ) \to H^0_\m(P) \to 0.
\end{equation}

As it is a submodule of $H^0_\m(H^{n-1}_I(R)_r) = 0$, $H^0_\m( Q) = 0$.  This fact, combined with the long exact sequence with respect to \eqref{SESeltFour: e}, shows that  $H^0_\m(P) \cong H^0_\m( H^{n-1}_I(R) )$.
This, \eqref{conclusion1LES: e}, and \eqref{SES12: e}, ensure the existence of a short exact sequence of the form 
\[
0 \to H^1_\m(H^{n-2}_I(R)) \to H^0_\m(H^{n-1}_J(R)) \to H^0_\m(H^{n-1}_I(R)) \to 0.
\]
If $E$ denotes the injective hull of the residue field of $R$ and $u =  \lambda_{1,2}(R/I)$, $v = \lambda_{0,1}(R/J)$, and $w =  \lambda_{0,1}(R/I)$, we have that $H^1_\m(H^{n-2}_I(R)) \cong E^{\oplus u}$, $H^0_\m(H^{n-1}_J(R)) \cong E^{\oplus v}$, and $H^0_\m(H^{n-1}_I(R)) \cong E^{\oplus w}$ \cite[Lemma 2.2]{LyuInvariants}. Then, 
the result follows.  
\end{proof}

For the Lyubeznik numbers on the superdiagonal of the Lyubeznik table not addressed by Proposition \ref{lambda12: P}, our analogous result is a bound.

\begin{proposition} \label{lambdaBound: P}
Suppose that $(R, \m)$ is an $n$-dimensional complete regular local ring containing a field.
Let $I$ be an ideal of $R$ for which 
$R/I$ is equidimensional of dimension $d \geq 3$. 
Given $2 \leq i \leq d - 2$, fix $r \in \m$ that satisfies the following property:
\begin{quotation} For $j = i$ and $j = i + 1$, if 
 $\Supp_R H^{n-j}_I(R) \neq \{ \m \}$, then $r$ is in no minimal prime of $\Supp_R H^{n-j}_I(R)$.
\end{quotation}
\noindent Then $\lambda_{i,i+1}(R/I) \geq \lambda_{i-1,i}(R/ ( I + \eltL r \eltR) )$. 
\end{proposition}

\begin{proof}
Let $\dim(R) = n$ and $J = I + \eltL r \eltR$, and consider the long exact sequence
\[
\cdots \to  H^{n-i-1}_I(R) \to H^{n-i-1}_I(R)_r \to H^{n-i}_J(R) \to H^{n-i}_I(R) \to H^{n-i}_I(R)_r \to \cdots. 
\]
Labeling kernels and images appropriately, we obtain the following short exact sequences, for $R$-modules $K,M,N$ and $P$:
\begin{align}
\label{EqA} 0 \to K \to &H^{n-i-1}_I(R) \to M \to 0 \\
\label{EqB} 0 \to M \to &H^{n-i-1}_I(R)_r \to N \to 0 \\
\label{EqC} 0 \to N \to &H^{n-i}_J(R) \to P \to 0.
\end{align}

As $K = H^0_{\eltL r \eltR} ( H^{n-i-1}_I(R))$
and $P=H^0_{\eltL r \eltR} ( H^{n-i}_I(R))$,  
\begin{equation} \label{supportRestrictions: e}
\dim \Supp_R (K) \leq i-1 \ \text{ and } \ \dim \Supp_R(P) \leq i-2
\end{equation}
by our choice of $r$ and Proposition \ref{nzdChoice: L}.
From \eqref{EqA}, we obtain the following long exact sequence in local cohomology with respect to $\m$.
\[
\cdots \to  H^i_\m(K) \to H^i_\m ( H^{n-i-1}_I(R)) \to H^i_\m (M) \to H^{i+1}_\m(K)  \to  \cdots. 
\]
Since $\InjDim(K) \leq \dim(K)\leq i-1$  by \eqref{supportRestrictions: e},  $H^j_\m(K)=0$ for $j\geq i$, so that, $H^i_\m( H^{n-i-1}_I(R))\cong H^i_\m(M).$
The short exact sequence \eqref{EqB} yields 
\[
\cdots \to H^j_\m ( H^{n-i-1}_I(R)_r ) \to H^j_\m (N) \to H^{j+1}_\m(M) \to H^{j+1}_\m ( H^{n-i-1}_I(R)_r ) \to  \cdots. 
\]
Since $r \in \m$,  $ H^j_\m ( H^{n-i-1}_I(R))_r  = 0$ for all integers $j$, hence $H^j_\m(N) \cong H^{j+1}_\m(M)$.  
From \eqref{EqC}, we obtain a long exact sequence of the from
\[
\cdots \to H^{i-2}_\m (P) \to H^{i-1}_\m(N) \to H^{i-1}_\m ( H^{n-i}_J(R) ) \to H^{i-1}_\m (P)  \to \cdots. 
\]
Since   $\InjDim(P) \leq \dim(P) \leq i-2,$ we have that $H^j_\m(P)=0$ for $j\geq i-1$.
Thus, the map from $H^{i-1}_\m(N)$ to $H^{i-1}_\m( H^{n-i}_J(R))$ is surjective.

Now, as
$H^{i-1}_\m(N) \cong H^i_\m(M) \cong H^i_\m( H^{n-i-1}_I(R))$, there exists a surjection 
\[H^i_\m( H^{n-i-1}_I(R)) \onto H^{i-1}_\m( H^{n-i}_J(R)).\]
We have that $H^i_\m(H^{n-i-1}_I(R)) \cong E^{\oplus u}$ and $H^{i-1}_\m(H^{n-i}_J(R)) \cong E^{\oplus v}$ \cite[Lemma 2.2]{LyuInvariants},
where $E$ denotes the injective hull of the residue field of $R$, $u = \lambda_{i,i+1}(R/I)$, and $v = \lambda_{i-1,i}(R/J)$.  
By applying Matlis duals $\Hom_R( -, E)$ to the surjection 
$E^{\oplus u} \onto  E^{\oplus v}$, 
we obtain an injection 
$R^{v} \hookrightarrow  R^{u}$
since $\Hom_R(E, E) \cong R$. 
Tensoring with the fraction field of $R$ and taking vector space dimensions, we conclude that 
$\lambda_{i-1,i}(R/J) = v \leq u = \lambda_{i,i+1}(R/I)$.
\end{proof}


\section{The graphs $\graph_\bullet$ and Lyubeznik numbers} \label{graphsLyu: S}


Recall that in Section \ref{graphNZD: S}, we compare the graphs associated to the ring to those associated to certain quotients modulo an element, and 
in Section \ref{LyuNZD: S}, we do the same for Lyubeznik numbers.
Using the results obtained therein, we can directly relate the Lyubeznik numbers and the graphs to one another.

\begin{lemma} \label{thetaD2Connected: L}
Let $A$ be an equidimensional complete local ring containing a field  of dimension $d$, with a separably closed residue field.  
If $d \leq 2$, then $\lambda_{1,2}(A) = 0$.
If $d \geq 3$ and $\graph_{d-2}(A)$ is connected, 
then $\lambda_{0,1}(A) =  \lambda_{1,2}(A) = 0$.  
\end{lemma}

\begin{proof}
It is known that if $d \leq 2$, then  $\lambda_{1,2}(A) = 0$ \cite[Section 1 and Proposition 2.2]{Walther01}.
Thus, let $d \geq 3$ and assume that $\graph_{d-2}(A)$ is connected.  Then $\graph_{d-1}(A)$ is connected as well; i.e., $\Spec^\circ(A)$ is connected and  $\lambda_{0,1}(A) = 0$.
Write $A \cong R/I$, where $I$ is an ideal of an $n$-dimensional regular local ring $(R, \m)$ containing a field.   
Fix a regular element $r \in \m$ satisfying the following property:  If $\Supp_R H^{n-2}_I(R) \neq \{ \m \}$, then 
$r$ is in no minimal prime of $\Supp_R H^{n-2}_I(R)$.  
Such an $r$ exists by prime avoidance since the $H^{n-2}_I(R)$ has finitely many associated primes \cite[Corollary 3.6]{Lyubeznik93}. 
Let $x$ denote the image of $r$ in $A \cong R/I$.
By Proposition \ref{graphModuloNZD: P}, $\graph_{d-2}(A/\eltL x \eltR)$ is connected, so that $\lambda_{0,1}(A/\eltL x \eltR) = 0$ by the second vanishing theorem of local cohomology
\cite{Ogus73}, \cite{PeskineSzpiro73}.
Thus, $\lambda_{1,2}(A) = 0$ by Proposition \ref{lambda12: P}.
\end{proof}

Recall that the Lyubeznik number $\lambda_{0,1}(A)$ of $A$ is one less than $\#\graph_{d-1}(A)  = \#\Spec^\circ(A)$ (see  \cite[Proposition 3.1]{Walther01} and the proof of \cite[Theorem 2.9]{HunekeLyubeznik90}).  
Motivated by this, the following lemma extends this relationship between Lyubeznik numbers and the graphs $\graph_t(A)$ associated to $A$, now comparing $\lambda_{1,2}(A)$ and $\graph_{d-2}(A)$.

\begin{proposition} \label{baseCaselambda12: P}
Let $A$ be an equidimensional complete  local ring containing a field of dimension $d \geq 3$, with a separably closed residue field. 
If $\graph_{d-1}(A)$ is connected, 
then 
$\lambda_{1,2}(A)=  \#\graph_{d-2}(A) - 1$.
\end{proposition}

\begin{proof}
The statement holds by Lemma \ref{thetaD2Connected: L} when $\graph_{d-2}(A)$ is connected.
Fix an integer $\conncomp \geq 2$, and by way of induction, assume that for all rings $B$ satisfying our hypotheses for which  $\#\graph_{d-2}(B) = \conncomp-1$, we have that 
$\lambda_{1,2}(B)=  \conncomp-2$.

Take a ring $A$ satisfying our hypotheses, for which $\# \graph_{d-2}(A) = \conncomp$.
Since $\graph_{d-1}(A)$ is connected and $\graph_{d-2}(A)$ is a subgraph with the same vertices, 
there exists an ordering $\Sigma_1, \Sigma_2, \ldots, \Sigma_{\conncomp}$ of the connected components of $\graph_{d-2}(A)$ 
for which the induced subgraph of $\graph_{d-1}(A)$ associated to  the union of the vertices of $\Sigma_1, \ldots, \Sigma_i$ is connected for all $1 \leq i \leq \conncomp$.

Write $A \cong R/I$, where $(R, \m)$ is an $n$-dimensional regular local ring containing a field, and $I$ is an ideal of $R$.  
For $1 \leq i \leq \conncomp$, let $\a_i$ be the ideal of $R$ that is the intersection of minimal primes of $I$ associated to the vertices of $\Sigma_i$, 
and let $\b_i = \a_1 \cap \cdots \cap \a_i$.
Since each $\graph_{d-1}(R/\a_i)$ and each $\graph_{d-1}(R/\b_i)$ is connected \cite[Theorem 2.9]{HunekeLyubeznik90} (see also its proof), 
\begin{equation} \label{HvanishingInduction: e}
H^{n-1}_{\a_i}(R) = H^{n-1}_{\b_i}(R) =0.
\end{equation}

For every minimal prime $\p$ of $\b_{\conncomp-1}$ and every minimal prime $\q$ of $\a_{\conncomp}$, since $\p A$ are $\q A$ are vertices of different connected components of $\graph_{d-2}(A)$, 
$\Ht_A(\p A + \q A) \geq d - 1$, so that $\Ht_R(\p + \q) \geq n - 1$. 
Hence $\Ht_R(\b_{\conncomp - 1} + \a_{\conncomp}) \geq n - 1$. 
Moreover, as $\Spec^\circ(A)$ is connected, $\sqrt{\b_{\conncomp-1} + \a_{\conncomp}} \neq \m$, so that 
\begin{equation} \label{heightJJ: e}
\Ht_R(\b_{\conncomp-1} + \a_{\conncomp}) = n-1.
\end{equation}
Therefore, since $R$ is regular, $H^{n-2}_{\b_{\conncomp-1}+\a_{\conncomp}}(R) = 0$.  
This fact, \eqref{HvanishingInduction: e}, and the Mayer-Vietoris sequence corresponding to $\b_{\conncomp-1}$ and $\a_{\conncomp}$, yield the short exact sequence
\begin{equation}  \label{SESinduction: e}
0 \to H^{n-2}_{\b_{\conncomp-1}}(R) \oplus H^{n-2}_{\a_{\conncomp}}(R) \to H^{n-2}_I(R) \to H^{n-1}_{\b_{\conncomp-1} + \a_{\conncomp}}(R) \to 0.
\end{equation}

Since $R/\b_{\conncomp-1}$ and $R/\a_{\conncomp}$ are equidimensional, $H^2_\m(H^{n-2}_{\b_{\conncomp-1}}(R)) = H^2_\m(H^{n-2}_{\a_{\conncomp}}(R)) = 0$ (see the proof of \cite[Property (4.4iii)]{Lyubeznik93}).
Moreover, by \eqref{heightJJ: e}, $R/(\b_{\conncomp-1} + \a_{\conncomp})$ is equidimensional of dimension one, 
so that $\lambda_{1,1}(R/(\b_{\conncomp-1} + \a_{\conncomp}))=1$ and $\lambda_{0,1}(R/(\b_{\conncomp-1} + \a_{\conncomp})) = 0$ (see, e.g.,  \cite[Section 1]{Walther01}).
In particular,  $H^0_\m(H^{n-1}_{\b_{\conncomp-1} + \a_{\conncomp}}(R)) = 0$  \cite[Lemma 1.4]{Lyubeznik93}.
Thus, the long exact sequence in local cohomology with support in $\m$ associated to \eqref{SESinduction: e} exhibits the short exact sequence 
\[
0 \to H^1_\m (H^{n-2}_{\b_{\conncomp-1}}(R)) \oplus H^1_\m(H^{n-2}_{\a_{\conncomp}}(R))  \to H^1_\m(H^{n-2}_I(R)) \to H^1_\m(H^{n-1}_{\b_{\conncomp-1}+ \a_{\conncomp}}(R)) \to 0.
\]
By Lyubeznik's work on $D$-modules \cite[Lemma 1.4]{Lyubeznik93}, the inductive hypothesis applied to $R/\b_{\conncomp-1}$, and Lemma \ref{thetaD2Connected: L}, 
\begin{align*}
\lambda_{1,2}(A) &= \lambda_{1,2}(R/\b_{\conncomp-1}) + \lambda_{1,2}(R/\a_{\conncomp}) + 1 = ( \conncomp - 2 ) + 0 + 1  = \conncomp- 1. 
\end{align*}

\vspace{-.7cm}

\end{proof}

\transition{
We can extend Lemma \ref{baseCaselambda12: P} to compare the other Lyubeznik numbers on the superdiagonal of the Lyubeznik table $\left[ \lambda_{i,j}(A) \right]_{0 \leq i, j \leq \dim(A)}$ with the number of 
connected components of a certain graph, but beyond $\lambda_{0,1}(A)$ and $\lambda_{1,2}(A)$, our equality is replaced with a bound.  
}

\begin{proposition}  \label{graphConnectedGeneral: P}
Let $A$ be an equidimensional complete local ring containing a field of dimension $d \geq 3$, with a separably closed residue field. 
If $\graph_{d-i}(A)$ is connected for some $1 \leq i \leq d-2$, then
$\lambda_{i,i+1}(A) \geq  \#\graph_{d-i -1}(A) - 1$.  
\end{proposition}

\begin{proof}
The case that $i=1$ holds by Lemma \ref{baseCaselambda12: P} (which, in particular, verifies the dimension two case).  
By way of induction on $i \geq 1$, assume that for any ring $B$ of dimension $e$ satisfying our hypotheses for which
$1 \leq i \leq e - 2$, and  $\graph_{e-i}(B)$ is connected, we have that $\lambda_{i,i+1}(B) \geq \#\graph_{e-i-1}(B)-1$.
 
Take a ring $A$ satisfying our hypotheses
for which $\dim(A) = d$, 
$2 \leq i + 1 \leq d - 2$ (so that $d \geq 4$), and
 $\graph_{d - i - 1}(A)$ is connected. 
By Proposition \ref{lambdaBound: P} and Theorem \ref{CCsNZD: T}, there exists $x$ in the maximal ideal of $A$ for which $\lambda_{i+1, i+2}(A) \geq \lambda_{i, i+1}(A/\eltL x \eltR)$, $\graph_{d-i-1}(A/\eltL x \eltR)$ is connected, 
and $ \#\graph_{d-i -2}(A) = \#\graph_{d-i-2}(A/\eltL x \eltR)$.       
Since  $\dim (A/\eltL x \eltR) = d  - 1 \geq 3$ and 
$
1 \leq i \leq d - 3, 
$
 we can apply the inductive hypothesis to $A/\eltL x \eltR$ and conclude that 
 \[
\lambda_{i+1, i+2}(A) \geq  \lambda_{i, i+1}(A/\eltL x \eltR) \geq \#\graph_{d-i -2}(A/\eltL x \eltR) - 1 = \#\graph_{d-i -2}(A) - 1.
\]

\vspace{-.7cm}

\end{proof}

Using induction and the Mayer-Vietoris sequence in local cohomology, we can remove the connectivity hypothesis in Lemmas \ref{baseCaselambda12: P} and \ref{graphConnectedGeneral: P}.  
Note that under the following theorem's hypotheses, $\lambda_{0,1}(A) = \#\graph_{d-1}(A)$ \cite[Proposition 3.1]{Walther01}, \cite[Proof of Theorem 2.9]{HunekeLyubeznik90}.

\begin{theorem}\label{MainTheorem: T}
Let $A$ be an equidimensional complete local ring containing a field of dimension $d \geq 3$, with a separably closed residue field.
Then 
\begin{enumerate}[label=\textup{(\arabic*)}]
\item \label{mainTheorem1: i}  $\lambda_{1,2}(A) = \#\graph_{d-2}(A) - \#\graph_{d-1}(A)$, and   
\item \label{mainTheorem2: i}   $\lambda_{i,i+1}(A) \geq  \#\graph_{d-i -1}(A)-  \#\graph_{d-i}(A)$ for all $1 \leq i \leq d - 2$.
\end{enumerate}

\end{theorem}

\begin{proof}
Fix $1 \leq i \leq d-2$.  The cases that $\#\graph_{d-i}(A)=1$, for $i=1$ and $i \geq 2$, hold by Lemmas \ref{baseCaselambda12: P} and \ref{graphConnectedGeneral: P}, respectively.  
We proceed by induction on $\#\graph_{d-i}(A) \geq 1$, for all $\#\graph_{d-i-1}(A) \geq \#\graph_{d-i}(A)$. 
Fix an integer $\conncomp_i \geq 2$, and assume that the statement holds for all rings $B$ satisfying our hypotheses, and for which $\# \graph_{d-i}(B) = \conncomp_i-1$.

Take a ring $A$ with $\# \graph_{d-i}(A)= \conncomp_i$, 
and for which $\# \graph_{d-i-1}(A) = \conncomp_{i+1} \geq \conncomp_i$.  
Write $A \cong R/I$, where $(R, \m)$ is an $n$-dimensional regular local ring containing a field, and $I$ is an ideal of $R$.  
Let $\a_i$ be the ideal of $R$ that is the intersection of the minimal primes of $I$ corresponding to the vertices of one connected component $\Sigma_i$ of $\graph_{d-i}(A)$, and let $\b_i$ denote the intersection of the remaining minimal primes of $R$.  
For every minimal prime $\p_i$ in $\Sigma_i$, and $\q_i$ of $A$ that is not a vertex of $\Sigma_i$, there is no edge between these vertices in $\graph_{d-i}(A)$, so $\Ht_A(\p_i + \q_i) \geq d - i + 1$.
Therefore, for any minimal prime $P_i$ of $\a_i$, and $Q_i$ of $\b_i$,  $\Ht_R(P_i + Q_i) \geq n - i + 1$. 
Thus, $\Ht_R(\a_i + \b_i) \geq n - i + 1$, 
and $H^j_{\a_i + \b_i}(R) = 0$ for $j \leq n - i $.  

The Mayer-Vietoris sequence in local cohomology with respect to $\a_i$ and $\b_i$, 
\[
\cdots \to H^{n-i-1}_{\a_i + \b_i}(R) \to H^{n-i-1}_{\a_i}(R) \oplus H^{n-i-1}_{\b_i}(R) \to H^{n-i-1}_I(R)  \to H^{n-i}_{\a_i + \b_i}(R) \to \cdots,
\]
now demonstrates that $H^{n-i-1}_{\a_i}(R) \oplus H^{n-i-1}_{\b_i}(R) \cong H^{n-i-1}_I(R)$.
Consequently, 
\[ H^i_\m (H^{n-i-1}_{\a_i}(R)) \oplus H^i_\m( H^{n-i-1}_{\b_i}(R)) \cong H^i_\m( H^{n-i-1}_I(R) ),\] and by  \cite[Lemma 1.4]{Lyubeznik93}, 
\begin{align} \label{sumLyuInduction: e}
\lambda_{i,i+1}(A) = \lambda_{i,i+1}(R/\a_i) + \lambda_{i,i+1}(R/\b_i).
\end{align} 

We know that $\graph_{d-i}(R/\a_i)$ is connected and $\#\graph_{d-i}(R/\b_i)=\conncomp_i-1$.  Moreover, $\#\graph_{d-i-1}(R/\a_i) + \#\graph_{d-i-1}(R/\b_i) = \conncomp_{i+1}$.   
Additionally noting \eqref{sumLyuInduction: e} and the inductive hypothesis, we conclude that, for $i=1$,
\begin{align*}
\lambda_{1,2}(A) 
= (\#\graph_{d-2}(R/\a_1) - 1 ) +  ( \#\graph_{d-2}(R/\b_1) - (\conncomp_1 -1 )) =  \conncomp_2 - \conncomp_1,
\end{align*}
and, likewise, for all $2 \leq i \leq d-2$,
\[
\lambda_{i,i+1}(A) 
\geq (\#\graph_{d-i-1}(R/\a_i) - 1 ) +  ( \#\graph_{d-i-1}(R/\b_i) - (\conncomp_i -1 )) 
= \conncomp_{i+1} - \conncomp_i.  
\]

\vspace{-.7cm}

\end{proof}

\transition{
Our example shows that the inequality in Theorem \ref{MainTheorem: T} \ref{mainTheorem2: i} can be strict when $i > 1$.  
}

\begin{example}\label{ExNotEq}
The nonzero Lyubeznik numbers of the ring $A$ appearing in Example \ref{graphintro: E} are $\lambda_{0,2}(A) = 1$, $\lambda_{2,3}(A) = 3$, and $\lambda_{4,4}(A) = 3$, so that 
\begin{align*}
\lambda_{1,2}(A) &=  \#\graph_{2}(A) - \#\graph_{3}(A)  = 0, \text{ but} \\
3 = \lambda_{2,3}(A)  &>  \#\graph_{1}(A) - \#\graph_{2}(A)  = 2.
\end{align*}
\end{example}


\section{Connectedness of spectra and Lyubeznik numbers}  \label{spectraLyu: S}


In this section, we combine results from Section \ref{graphsLyu: S}, which relate the Lyubeznik numbers and our graphs, 
and those from Section \ref{spectraGraph: S}, which relate our graphs to connectedness properties of spectra, in order to 
prove that the Lyubeznik numbers determine connectedness properties of spectra.

Calling upon these results, we also recover bounds between connectedness dimension, cohomological dimension, and depth \cite{FalNagoya, FalPRIMS, HochsterHuneke94, Varbaro09}.

As mentioned in the introduction, in our setting it is well known that if $\dim(A) \geq 2$, then the connectedness dimension of $A$ is at least one if and only if the Lyubeznik number $\lambda_{0,1}(A)$ vanishes.  The following analogue characterizes when
connectedness dimension is at least two, again via the vanishing of Lyubeznik numbers.

\begin{theorem} \label{lambda12Final: T}
Let $A$ be an equidimensional complete local ring containing a field of dimension $d \geq 3$, with a separably closed residue field. 
Then 
$\cc(A) \geq 2$ if and only if $\lambda_{0,1}(A) = \lambda_{1,2}(A) = 0$.  
In particular, if $\Spec^\circ(A)$ is connected, then $\cc(A) \geq 2$ if and only if $\lambda_{1,2}(A) = 0$.  
\end{theorem}

\begin{proof}
By Proposition \ref{graphConn: P}, we know that $\cc(A) \geq 2$ if and only if $\graph_{d-2}(A)$ is connected.  
If $\graph_{d-2}(A)$ is connected, then so is the supergraph $\graph_{d-1}(A)$.
Moreover, $\graph_{d-1}(A)$ is connected precisely if $\Spec^\circ(A)$ is connected; i.e.,  $\lambda_{0,1}(A) = 0$.
The result now follows since in this case, $\graph_{d-2}(A)$ is connected if and only if $\lambda_{1,2}(A) = 0$ by Lemma \ref{baseCaselambda12: P}.
\end{proof}

Theorem \ref{lambda12Final: T} allows us to conclude that for rings of dimension at most three satisfying our hypotheses, 
the connectedness dimension is completely determined by the Lyubeznik numbers.

In general, the vanishing of the Lyubeznik numbers on the superdiagonal of the Lyubeznik table determines a bound on the connectedness dimension.  
We apply the following lemma to obtain this bound in Theorem \ref{vanishingLyuNumConnDim: T}.

\begin{lemma} \label{connDimUp: L}
Let $A$ be an equidimensional complete local ring containing a field of dimension $d \geq 2$, with a separably closed residue field. 
If for some $0 \leq i \leq d - 2$, $\cc(A) \geq i$ and $\lambda_{i,i+1}(A) = 0$, then $\cc(A) \geq i + 1$.  
\end{lemma}

\begin{proof}
First notice that the statement holds when $i=0$ since we always have that $\cc(A) \geq 0$, and if $\lambda_{0,1}(A) = 0$, then $\Spec^\circ(A)$ is connected, so that $\cc(A) \geq 1$.   

Now consider the case that $1 \leq i \leq d -2$, so that $d \geq 3$.
By Proposition \ref{graphConn: P}, $\cc(A) \geq i$ if and only if $\graph_{d-i}(A)$ is connected.  Therefore, by Lemma \ref{graphConnectedGeneral: P}, 
$\#\graph_{d-i-1}(A) \leq \lambda_{i, i+1}(A) + 1 = 1$; i.e., $\graph_{d-i-1}(A)$ is connected, and $\cc(A) \geq i + 1$ by Proposition \ref{graphConn: P}.
\end{proof}

\begin{remark}
If, under the hypotheses of Lemma \ref{connDimUp: L}, $\lambda_{i, i+1}(A) = 0$ for all $0 \leq i \leq d-2$, then the connectedness dimension is at least $d-1$, and equals  $d$ precisely if $A$ has only one minimal prime.

Indeed, write $A \cong R/I$, where $I$ is an ideal of $(R, \m)$, a $n$-dimensional regular local ring containing a field.  
In the spectral sequence $E_{2}^{p,q} = H^p_\m(H^q_I(R)) \underset{p}{\implies} H^{p+q}_\m(R)$
\cite[Proposition 1.4]{Hartshorne67}, 
for each $r \geq 2$, the only possibly nonzero differential to $E_r^{d, n-d}$ is 
$\partial_r^{d-r, n-d+r-1}: E_r^{d-r, n-d+r-1} \to E_r^{d, n-d}$.  However, $E_2^{d-r, n-d+r-1} = H^{d-r}_\m(H^{n-d+r-1}_I(R)) = 0$ since $\lambda_{d-r, d - r + 1}(A) = 0$.  
Let $E$ denote the injective hull of the residue field of $R$.
Then since $E_{2}^{p,q}$ is only nonzero if $p + q = n$ if $p = d$ and $q = n-d$, and
$H^n_\m(R) \cong E$, we must have that $E_{2}^{d,n-d} = H^d_\m(H^{n-d}_I(R)) \cong E$.  Thus, $\lambda_{d,d}(A) = 1$, so that the Hochster-Huneke graph of $A$ is connected, and $\cc(A) \geq d-1$ \cite[Theorem 1.3]{LyuInvariants}, \cite[Main Theorem]{Zhang07}, \cite[Theorem 3.6]{HochsterHuneke94}.
Since $\cc(A)$ and $\dim(A)$ coincide if and only if $A$ has only one minimal prime \cite[Remark 19.2.5]{BrodmannSharpEd2}, the connectedness dimension is determined. 

The following theorem establishes a stronger relationship between the vanishing of Lyubeznik numbers and connectedness dimension. 
\end{remark}

\begin{theorem} \label{vanishingLyuNumConnDim: T}
Let $A$ be an equidimensional complete local ring containing a field of dimension $d \geq 2$, with a separably closed residue field. 
If for some $0 \leq i \leq d - 2$, 
$\lambda_{j, j+1}(A) = 0$ for all $0 \leq j \leq i$, then $\cc(A) \geq i+1$. 
\end{theorem}

\begin{proof}
We proceed by induction to show that $\cc(A) \geq j  + 1$ for all $j \leq i$, so that $\cc(A) \geq i + 1$.
Indeed, if $\lambda_{0,1}(A) = 0$, then $\cc(A) \geq 1$, and if $\cc(A) \geq j+1$ for some $j \leq i$, then $\cc(A) \geq j + 2$ by Lemma \ref{connDimUp: L}.  
\end{proof}

\begin{example}
For the ring from Example \ref{graphintro: E}, $\lambda_{0,1}(A) = \lambda_{1,2}(A) = 0$ and $\lambda_{2,3}(A) = 3$.
Thus, in this case, Theorem \ref{vanishingLyuNumConnDim: T} is sharp in the sense that \[ \cc(A) = 2 = \min\{i \geq 0 \mid \lambda_{i, i+1}(A) \neq 0\}. \]
\end{example}

\transition{We recover a strengthening of Faltings' connectedness theorem that relates connectedness and cohomological dimensions.
Given an ideal $I$ of a Noetherian ring $R$, recall that the \emph{cohomological dimension} of $I$ in $R$, denoted $\cd(I, R)$, is the maximum index $i \geq 0$ for which 
the local cohomology module $H^i_I(R)$ is not zero.
}

\begin{corollary}[{cf.\,\cite{FalNagoya}, \cite{FalPRIMS}, \cite[Theorem 3.3]{HochsterHuneke94}, \cite[Corollary 1.7]{Varbaro09}}] \label{CDCD: C}
Let $R$ be an $n$-dimensional complete regular local ring containing a field, with a separably closed residue field.
If $I$ is an ideal of $R$ for which $R/I$ is equidimensional and of dimension $d \geq 2$, then 
\[
\cc(R/I) \geq n - \cd(I, R) - 1.
\]
\end{corollary}

\begin{proof}
First assume that $\lambda_{i, i+1}(R/I) = 0$ for all $0 \leq i \leq d-2$.
Since $\cd(I, R) \geq n - d$, the depth of $R$ on $I$, 
we have that $\cc(A) \geq d - 1 \geq n - \cd(I, R) -1$ by Theorem \ref{vanishingLyuNumConnDim: T}.   

Otherwise, let $\counter = \min\{ 0 \leq i \leq d-2 \mid \lambda_{i, i + 1}(A) \neq 0 \}$. 
For $0 \leq j < n - \cd(I, R)  - 1$, since $n - j - 1 > \cd(I, R)$, we have that $H^{n-j-1}_I(R) = 0$, so that $\lambda_{j, j+1}(R/I) = 0$.
Therefore, $\counter \geq n - \cd(I, R)  - 1$, 
which is at most $\cc(R/I)$ by Theorem \ref{vanishingLyuNumConnDim: T}. 
\end{proof}

Our results also allow us to immediately recover the following inequality between connectedness dimension and depth established by Varbaro.  
This can be thought of as a ``hidden application'' of the Lyubeznik numbers since neither they, nor local cohomology, appears in its statement.

\begin{corollary}[{cf.\,\cite[Proposition 2.11]{Varbaro09}}]  \label{CDdepth: C}  
Let $A$ be an equidimensional complete local ring containing a field, with a separably closed residue field.  
Then 
\[
\cc(A) \geq \depth(A) - 1. 
\]
\end{corollary}

\begin{proof}
The statement trivially holds when $\dim(A) \leq 1$ or   $\cc(A) \geq \dim(A) -1$, so assume that
$\dim(A) \geq 2$ and $\cc(A) \leq \dim(A) -2.$  
In this case, the result follows from Theorem \ref{vanishingLyuNumConnDim: T}, since
 $\lambda_{i, i+1}(A)  = 0$ for $0 \leq i \leq \depth(A) - 2$ \cite[Corollary 3.3]{Varbaro13}. 
\end{proof}

\begin{remark}
Let $A$ be an equidimensional complete local ring of dimension $d \geq 2$ containing a field, with a separably closed residue field.  
Assume that $\lambda_{j , j + 1}(A) = 0$ for $0 \leq j \leq i$ for some $i \leq d - 2$.
Applying Varbaro's result that  $\lambda_{i,i+1}(A) = 0$ for $i < \depth(A) - 1$ \cite[Corollary 3.3]{Varbaro13}, 
our proof of Corollary \ref{CDdepth: C} actually shows that 
\[
\cc(A) \geq i + 1 \geq \depth(A) - 1.
\]
Strikingly, this means that the value $i+1$ is a numerical invariant of $A$ that always sits between the connectedness dimension and the depth of $A$, providing more refined information about the relationship between these two numbers. 
\end{remark}

\begin{remark}
In Theorem \ref{vanishingLyuNumConnDim: T} and Corollary \ref{CDdepth: C}, it is sufficient to assume that $A$ is a  formally equidimensional local ring containing a field, with a separably closed residue field; e.g., an equidimensional quotient of a regular local ring.  
Indeed, $\cc(A) \geq \cc(\widehat{A})$, $\lambda_{i,j}(A) = \lambda_{i, j}(\widehat{A})$, and $\depth(A) \leq \depth(\widehat{A})$ \cite[Lemma 19.3.1]{BrodmannSharpEd2},  \cite[Lemma 3.2]{Lyubeznik93}. 
\end{remark}


\begin{proposition} \label{countConnComp: P}
Let $A$ be an equidimensional complete local ring containing a field of dimension $d \geq 2$, with a separably closed residue field.
Then for all $0 \leq t \leq d-2$,
\[
\max\{  \#\(\Spec(A) \setminus X \)  \mid X \subseteq \Spec(A) \textup{ closed, } \dim(X) \leq t \} \leq 1 + \sum \limits_{i =0}^t \lambda_{i,i+1}(A),
\]
with equality when $t = 0$ or $1$.  
\end{proposition}

\begin{proof}
By Theorem \ref{MainTheorem: T}, for $1 \leq t \leq d - 2$, 
\begin{align*}
1 + \sum_{i=0}^t \lambda_{i,i+1}(A) &\geq 1 + \lambda_{0,1}(A) +  \sum_{i=1}^t \( \#\graph_{d- i -1}(A)-  \#\graph_{d-i}(A) \) \\
&= \#\graph_{d-1} + \(  \#\graph_{d- t -1}(A) -  \#\graph_{d -1}(A) \) = \#\graph_{d- t -1}(A),
\end{align*}
with equality for $t=1$.  Since $ \#\Spec^\circ(A) = \lambda_{0,1}(A) + 1$, we also have equality for $t=0$.  
We are done by Corollary \ref{numCC: C}.  
\end{proof}

\begin{remark}
Proposition \ref{countConnComp: P} is a strengthening of Theorem \ref{vanishingLyuNumConnDim: T}:
If $\lambda_{i,i+1}(A) = 0$ for all $0 \leq i \leq t \leq d - 2$, then the right-hand side of the equation in the statement equals one, so the left-hand side of the equation must also equal one,
and $\cc(A) \geq d - t$.   
\end{remark}

Our running example illustrates that the bound in Proposition \ref{countConnComp: P} can be strict.

\begin{example}
Taking the ring $A$ from Example \ref{graphintro: E}, we have that
$\#\graph_1(A) = 3$ and $1 + \lambda_{0,1}(A) + \lambda_{1,2}(A) + \lambda_{2,3}(A) = 4$, so that the 
inequality is strict in Proposition \ref{countConnComp: P} when $t=2$.  
\end{example}

\begin{remark} \label{dimensionThree: R}
Proposition \ref{countConnComp: P} allows one to describe the Lyubeznik table of $A$ when $d=3$.  
Indeed, write $A \cong R/I$, where $I$ is an ideal of $(R, \m)$, a regular local ring containing a field.
A straightforward argument using the spectral sequence $E_{2}^{p,q} = H^p_\m(H^q_I(R)) \underset{p}{\implies} H^{p+q}_\m(R)$ shows that $\lambda_{0,2}(A) = \lambda_{1,3}(A)$, $\lambda_{3,3}(A) = \lambda_{0,1}(A) + \lambda_{1,2}(A) + 1$, and all other Lyubeznik numbers must vanish. 
It is known that $\lambda_{0,1}(A) = \#\Spec^\circ(A)-1$, and $\lambda_{3,3}(A)$ is the maximum number of connected components of $\Spec(A) \setminus \V(\a)$ among all ideals $\a$ of $A$ that are either its maximal ideal, or have height two  \cite[Proposition 3.1]{Walther01},  \cite[Theorem 3.6]{HochsterHuneke94}, \cite[Theorem 1.3]{LyuInvariants}, \cite[Main Theorem]{Zhang07}.
Therefore, $\lambda_{1,2}(A)$ must equal the maximum number of connected components of $\Spec(A) \setminus \V(\a)$ among all ideals $\a$ of $A$ of height two.
\end{remark}

\begin{example}
It is not difficult to find rings $A$ of dimension three with $\lambda_{3,3}(A) = 3$, and for which $\lambda_{0,1}(A)$ and $\lambda_{1,2}(A)$ take on any of the possible values determined by Proposition \ref{countConnComp: P} and Remark \ref{dimensionThree: R}.  
For instance, given a field $\k$, let   
\begin{align*}
 S = \k\llbracket x_1, \ldots, x_9 \rrbracket / \eltL x_1, \ldots, x_6 \eltR \cap \eltL x_4, \ldots, x_9 \eltR,
\end{align*}
and fix the following ideals of $S$: 
\begin{alignat*}{3}
\p = \eltL x_1, x_2, x_3, x_7, x_8, x_9 \eltR & \hspace{.3cm} & \q =  \eltL x_2, x_3, x_4, x_7, x_8, x_9 \eltR   & \hspace{.3cm} &
\r =   \eltL x_1, x_2, x_4, x_5, x_7, x_8 \eltR
\end{alignat*}
Using Proposition \ref{countConnComp: P}, we can compute the following values:  
\begin{alignat*}{4}
&\cc(S/\p) = 1 & \hspace{.5cm} & \cc(S/\q) = 1 & \hspace{.5cm} &\cc(S/\r) = 2& \\
&\lambda_{0,1}(S/\p) = 2 & \hspace{.5cm} & \lambda_{0,1}(S/\q) = 1 & \hspace{.5cm} &\lambda_{0,1}(S/\r) = 0& \\
 &\lambda_{1,2}(S/\p) = 0 & \hspace{.5cm} & \lambda_{1,2}(S/\q) = 1 & \hspace{.5cm} &\lambda_{1,2}(S/\r) = 2&
\end{alignat*}  
\end{example}

\section{Lyubeznik numbers of projective varieties}

The general setting for this section is the following. 
\begin{setting} \label{proj: Setting}
Let $X$ be an equidimensional projective variety of dimension $d$ over a field $\k$. 
By choosing an embedding $X \hookrightarrow \PP^n_\k$, one can write $X=\Proj(R/I)$, where $I$ is a homogeneous ideal of the polynomial ring 
 $R=\k[x_0,\ldots,x_n]$. 
 Let $\m = \eltL x_0, \ldots, x_n \eltR$ be the homogeneous maximal ideal of $R$, and $A = ( R/I )_\m$ the local ring at the vertex of the affine cone over $X$.  
 \end{setting}
 
In our setting, Lyubeznik posed the question of whether the Lyubeznik numbers $\lambda_{i,j}(A)$ are independent of $n$, and the choice of embedding of $X$ into $ \PP^n_\k$.
\cite[Open question, Section 7]{LyuSurveyLC}. 
The question has been answered affirmatively for all Lyubeznik numbers by Zhang when $\k$ has prime characteristic \cite[Theorem 1.1]{WZ-projective}, 
and by Switala when $X$ is smooth \cite[Main Theorem 1.2]{SwitalaNonsingular}.

The highest Lyubeznik number $\lambda_{d+1, d+1}(A)$ is independent of the choice of embedding  \cite[Theorem 2.7]{Zhang07}, and it is 
well known that  $\lambda_{0,1}(A)$ is as well (cf.\,\cite[Proposition 3.1]{Walther01}).  
Using Theorem \ref{MainTheorem: T}, we can prove the same for $\lambda_{1,2}(A)$, and explicitly characterize its value.
To do so, we use a graph defined analogously to that in Definition \ref{graph: D}.

\begin{definition}[cf.\,\,{Definition \ref{graph: D}}, {\cite{WZ-projective}}]
For $X$ an equidimensional projective variety of dimension $d$, given an integer $1 \leq t \leq d$, define the graph $\graph_t(X)$ as follows:
\begin{enumerate}[label=\textup{(\arabic*)}]
\item The vertices of $\graph_t(X)$ are indexed by the irreducible components of $X$,  and 
\item There is an edge between distinct vertices $Z$ and $W$ if and only if 
\[\dim( Z \cap W) \geq d - t.\]
\end{enumerate}
\end{definition}

\begin{lemma} \label{graphsProj: L}
Under the assumptions in Setting \ref{proj: Setting},  for  all integers $1 \leq t \leq d$, 
\[\#\graph_t(X) = \#\graph_t(\widehat{A}).\]
\end{lemma}

\begin{proof}
Let $\overline{R} = R/I$, and let $Z_1,\ldots,Z_s$ denote the irreducible components of $X$.
Then there exist homogeneous prime ideals $\p_1,\ldots,\p_s$ of $\overline{R}$ such that
$Z_i\cong \Proj(\overline{R}/\p_i)$ via the correspondence $X=\Proj(\overline{R})$.
Furthermore, $\p_1,\ldots,\p_s$ are the minimal primes of $\overline{R}$.
Since each $\p_i$ is a homogeneous minimal prime ideal of $\overline{R}$, each $\p_i \widehat{A}$ is a minimal  prime ideal of $\widehat{A}$. 
This correspondence gives us a bijection between the vertices of 
$\graph_t(X)$ and the vertices of $\graph_t(\widehat{A})$.  
Furthermore, for all $1 \leq i, j \leq s$, 
\begin{align*}
\Ht_{\widehat{A}}(\p_i\widehat{A}+\p_j\widehat{A})&=\Ht_{\widehat{A}}((\p_i+\p_j)\widehat{A})\\
&=\Ht_{\overline{R}}(\p_i+\p_j)\\
&=\dim(\overline{R})-\dim\(\overline{R}/(\p_i+\p_j) \) \\
&=(d+1)-(\dim(Z_i\cap Z_j)+1)\\
&=d - \dim(Z_i\cap Z_j),
\end{align*}
so that 
$\Ht_{\widehat{A}}(\p_i\widehat{A}+\p_j\widehat{A})\leq t$
if and only if
$\dim(Z_i\cap Z_j)\geq d-t.
$; i.e., 
there is an edge between $\p_i \widehat{A}$ and $\p_j \widehat{A}$ in $\graph_t(\widehat{A})$
if and only if there is an edge between $Z_i$ and $Z_j$ in $\graph_t(X)$.
\end{proof}

%

When $A$ is a complete local ring of dimension at most two that contains a field, and has a separably closed residue field, $\lambda_{1,2}(A) = 0$ \cite[Proposition 3.1]{Walther01}.
Therefore, to establish independence of embedding, we must only consider projective varieties of dimension at least two.

\begin{theorem} \label{lambda12Proj: T}
Under the assumptions in Setting \ref{proj: Setting}, if $d \geq 2$, then
\[
\lambda_{1,2}(A)=\#\graph_{d-1}(X\otimes_\k\overline{\k})-\#\graph_d(X\otimes_\k\overline{\k}).
\]
Consequently, $\lambda_{1,2}(A)$ does not depend on the choice of embedding, regardless of the dimension of $A$. 
\end{theorem}

\transition{
Before proceeding to prove Theorem \ref{lambda12Proj: T}, we point out that in light of Example \ref{ExNotEq}, one cannot expect to show the independence of $\lambda_{i,i+1}(X)$ for $i > 1$
via the graphs $\graph_t(X\otimes_\k\overline{\k})$. 
However, we establish a lower bound for these numbers that is independent of the embedding
of $X$ into a projective space.  
}

\begin{theorem} \label{LyuProjBound: T} 
Under the assumptions in Setting \ref{proj: Setting}, for all $2 \leq i \leq d-1$, 
\[
\lambda_{i,i+1}(A) \geq\#\graph_{d-i-1}(X\otimes_\k\overline{\k})-\#\graph_{d-i}(X\otimes_\k\overline{\k}).
\]
\end{theorem}

\begin{proof}[Proof of Theorems \ref{lambda12Proj: T} and \ref{LyuProjBound: T}]
We can assume, without loss of generality, that $\k$ is algebraically closed since $X\otimes_\k\overline{\k}=\Proj(R/I\otimes_\k\overline{\k})$ and 
$\lambda_{i,j}(A)= \lambda_{i,j}(\widehat{A}) = \lambda_{i,j}(\widehat{A}\otimes_\k\overline{\k})$ for all $i, j \in \mathbb{Z}$.  
By Lemma \ref{graphsProj: L}, for each $\#\graph_t(X\otimes_\k\overline{\k}) = \#\graph_t(\widehat{A}\otimes_\k\overline{\k})$ for each  $1 \leq t \leq d$, so that 
the conclusion holds by Theorem \ref{MainTheorem: T}.
\end{proof}

\section*{Acknowledgments}
The authors thank Daniel Hern\'andez, who provided valuable mathematical insight, including through his questions.
We also thank Matteo Varbaro for pointing out some interesting implications of our work.  
Moreover, we are grateful for support of the research herein.   
N\'u\~nez-Betancourt was partially supported by NSF DMS \#1502282, 
Spiroff by Simons Foundation Collaboration Grant 245926, and Witt by NSF DMS \#1501404/1623035.

\bibliographystyle{alpha}
\bibliography{References}

\begin{thebibliography}{HPNBW17}

\bibitem[BS13]{BrodmannSharpEd2}
M.~P. Brodmann and R.~Y. Sharp.
\newblock {\em Local cohomology}, volume 136 of {\em Cambridge Studies in
  Advanced Mathematics}.
\newblock Cambridge University Press, Cambridge, second edition, 2013.
\newblock An algebraic introduction with geometric applications.

\bibitem[DT16]{DaoTakagi}
Hailong Dao and Shunsuke Takagi.
\newblock On the relationship between depth and cohomological dimension.
\newblock {\em Compos. Math.}, 152(4):876--888, 2016.

\bibitem[Fal80a]{FalNagoya}
Gerd Faltings.
\newblock A contribution to the theory of formal meromorphic functions.
\newblock {\em Nagoya Math. J.}, 77:99--106, 1980.

\bibitem[Fal80b]{FalPRIMS}
Gerd Faltings.
\newblock Some theorems about formal functions.
\newblock {\em Publ. Res. Inst. Math. Sci.}, 16(3):721--737, 1980.

\bibitem[Gro68]{GrothendieckSGA2}
Alexander Grothendieck.
\newblock {\em Cohomologie locale des faisceaux coh\'erents et th\'eor\`emes de
  {L}efschetz locaux et globaux {$(SGA$} {$2)$}}.
\newblock North-Holland Publishing Co., Amsterdam; Masson \& Cie, \'Editeur,
  Paris, 1968.
\newblock Augment\'e d'un expos\'e par Mich\`ele Raynaud, S\'eminaire de
  G\'eom\'etrie Alg\'ebrique du Bois-Marie, 1962, Advanced Studies in Pure
  Mathematics, Vol. 2.

\bibitem[Har62]{Hartshorne62}
Robin Hartshorne.
\newblock Complete intersections and connectedness.
\newblock {\em Amer. J. Math.}, 84:497--508, 1962.

\bibitem[Har67]{Hartshorne67}
Robin Hartshorne.
\newblock {\em Local cohomology}, volume 1961 of {\em A seminar given by A.
  Grothendieck, Harvard University, Fall}.
\newblock Springer-Verlag, Berlin, 1967.

\bibitem[Har68]{Hartshorne68}
Robin Hartshorne.
\newblock Cohomological dimension of algebraic varieties.
\newblock {\em Ann. of Math. (2)}, 88:403--450, 1968.

\bibitem[HH94]{HochsterHuneke94}
Melvin Hochster and Craig Huneke.
\newblock Indecomposable canonical modules and connectedness.
\newblock In {\em Commutative algebra: syzygies, multiplicities, and birational
  algebra (South Hadley, MA, 1992)}, volume 159 of {\em Contemp. Math.}, pages
  197--208. Amer. Math. Soc., Providence, RI, 1994.

\bibitem[HL90]{HunekeLyubeznik90}
Craig Huneke and Gennady Lyubeznik.
\newblock On the vanishing of local cohomology modules.
\newblock {\em Invent. Math.}, 102(1):73--93, 1990.

\bibitem[HPNBW17]{HNBPWInjDim}
Daniel~J. Hern{\'a}ndez, Felipe P{\'e}rez, Luis N{\'u}{\~n}ez-Betancourt, and
  Emily~E. Witt.
\newblock Lyubeznik numbers and the injective dimension of local cohomology in
  mixed characteristic.
\newblock {\em to appear in Trans. Amer. Math. Soc.}, 2017.

\bibitem[HS93]{Huneke}
Craig~L. Huneke and Rodney~Y. Sharp.
\newblock Bass numbers of local cohomology modules.
\newblock {\em Trans. Amer. Math. Soc.}, 339(2):765--779, 1993.

\bibitem[KLZ16]{ExtHar}
Mordechai Katzman, Gennady Lyubeznik, and Wenliang Zhang.
\newblock An extension of a theorem of {H}artshorne.
\newblock {\em Proc. Amer. Math. Soc.}, 144(3):955--962, 2016.

\bibitem[Lyu93]{Lyubeznik93}
Gennady Lyubeznik.
\newblock Finiteness properties of local cohomology modules (an application of
  {$D$}-modules to commutative algebra).
\newblock {\em Invent. Math.}, 113(1):41--55, 1993.

\bibitem[Lyu02]{LyuSurveyLC}
Gennady Lyubeznik.
\newblock A partial survey of local cohomology.
\newblock In {\em Local cohomology and its applications ({G}uanajuato, 1999)},
  volume 226 of {\em Lecture Notes in Pure and Appl. Math.}, pages 121--154.
  Dekker, New York, 2002.

\bibitem[Lyu06]{LyuInvariants}
Gennady Lyubeznik.
\newblock On some local cohomology invariants of local rings.
\newblock {\em Math. Z.}, 254(3):627--640, 2006.

\bibitem[Lyu07]{LyuLC}
Gennady Lyubeznik.
\newblock On some local cohomology modules.
\newblock {\em Adv. Math.}, 213(2):621--643, 2007.

\bibitem[Ogu73]{Ogus73}
Arthur Ogus.
\newblock Local cohomological dimension of algebraic varieties.
\newblock {\em Ann. of Math. (2)}, 98:327--365, 1973.

\bibitem[PS73]{PeskineSzpiro73}
C.~Peskine and L.~Szpiro.
\newblock Dimension projective finie et cohomologie locale. {A}pplications \`a
  la d\'emonstration de conjectures de {M}. {A}uslander, {H}. {B}ass et {A}.
  {G}rothendieck.
\newblock {\em Inst. Hautes \'Etudes Sci. Publ. Math.}, (42):47--119, 1973.

\bibitem[Swi15]{SwitalaNonsingular}
Nicholas Switala.
\newblock Lyubeznik numbers for nonsingular projective varieties.
\newblock {\em Bull. Lond. Math. Soc.}, 47(1):1--6, 2015.

\bibitem[Var09]{Varbaro09}
Matteo Varbaro.
\newblock Gr\"{o}bner deformations, connectedness and cohomological dimension.
\newblock {\em J. Algebra}, 322(7):2492--2507, 2009.

\bibitem[Var13]{Varbaro13}
Matteo Varbaro.
\newblock Cohomological and projective dimensions.
\newblock {\em Compos. Math.}, 149(7):1203--1210, 2013.

\bibitem[Wal01]{Walther01}
Uli Walther.
\newblock On the {L}yubeznik numbers of a local ring.
\newblock {\em Proc. Amer. Math. Soc.}, 129(6):1631--1634 (electronic), 2001.

\bibitem[Zha07]{Zhang07}
Wenliang Zhang.
\newblock On the highest {L}yubeznik number of a local ring.
\newblock {\em Compos. Math.}, 143(1):82--88, 2007.

\bibitem[Zha11]{WZ-projective}
Wenliang Zhang.
\newblock Lyubeznik numbers of projective schemes.
\newblock {\em Adv. Math.}, 228(1):575--616, 2011.

\end{thebibliography}


\vspace{.5cm}

{\footnotesize

\textsc{Centro de Investigaci\'on en Matem\'aticas, Jalisco S/N, Col. Valenciana CP: 36023, 
Guanajuato, Gto., M\'exico} \\ \indent \emph{Email address}:  {\tt luisnub@cimat.mx} 

\vspace{.25cm}

\textsc{Department of Mathematics, University of Mississippi, 305 Hume Hall, P.O. Box 1848,  University, MS 38677}  \\ 
\indent \emph{Email address}: {\tt spiroff@olemiss.edu}

\vspace{.25cm}

\textsc{Department of Mathematics, University of Kansas, 405 Snow Hall, 1460 Jayhawk Blvd, Lawrence, KS 66045} \\ 
\indent \emph{Email address}: {\tt witt@ku.edu}
}

\end{document}